\documentclass[10pt]{amsart}
\usepackage{amsfonts}
\usepackage{graphics,graphicx}
\usepackage{caption} 
\usepackage{epsfig}
\usepackage[utf8]{inputenc}
\usepackage[T1]{fontenc}
\usepackage{color}
\usepackage{amssymb}
\usepackage{mathrsfs}
\usepackage{amsthm}
\usepackage{amsmath}
\usepackage{amssymb}
\usepackage{latexsym}
\usepackage{fig4tex184}
\usepackage{bbold}
\makeatletter
\@addtoreset{equation}{section} 
\makeatother
\newtheorem{cor}{Corollary}[section]
\newtheorem{thm}{Theorem}[section]
\newfont{\sBlackboard}{msbm10 scaled 900}

\newcommand{\dd}     {{\rm d}}
\newcommand{\mylabel}[1]{\label{#1}
            \ifx\undefined\stillediting
            \else \fbox{$#1$}\fi }
\newcommand{\BE}{\begin{equation}}

\newcommand{\EEQ}{\end{equation}}
\newcommand{\rfb}[1]{\mbox{\rm
   (\ref{#1})}\ifx\undefined\stillediting\else:\fbox{$#1$}\fi}

\newfont{\Blackboard}{msbm10 scaled 1200}
\newcommand{\bl}[1]{\mbox{\Blackboard #1}}
\newfont{\roma}{cmr10 scaled 1200}

\def\CC{\rm \hbox{C\kern-.56em\raise.4ex
         \hbox{$\scriptscriptstyle |$}\kern+0.5 em }}
\newcommand{\ud}{\mathrm{d}}
\newcommand{\be}{\begin{equation}}
\newcommand{\ee}{\end{equation}}
\newcommand{\beq}{\begin{eqnarray}}
\newcommand{\eeq}{\end{eqnarray}}
\newcommand{\beqs}{\begin{eqnarray*}}
\newcommand{\eeqs}{\end{eqnarray*}}
\newcommand{\bt}{\begin{thm}}
\newcommand{\et}{\end{thm}}
\newcommand{\br}{\begin{remark}}
\newcommand{\er}{\end{remark}}
\newcommand{\bc}{\begin{cor}}
\newcommand{\ec}{\end{cor}}
\newcommand{\el}{\end{lem}}
\newcommand{\bd}{\begin{definition}}
\newcommand{\ed}{\end{definition}}

\newcommand{\mm}    {{\hbox{\hskip 0.5pt}}}

\newcommand{\bluff} {{\hbox{\raise 15pt \hbox{\mm}}}}
\makeatletter
\def\section{\@startsection {section}{1}{\z@}{-3.5ex plus -1ex minus
    -.2ex}{2.3ex plus .2ex}{\large\bf}}
\makeatother
\def\be{\begin{equation}}
\def\ee{\end{equation}}

\def\ds{\displaystyle}

\newcommand{\im}{\mathrm{Im}}
\newcommand{\e}{\mathrm{e}}

\newcommand{\R}{\bl{R}}
\newcommand{\N}{\bl{N}}
\newcommand{\op}{\mathrm{op}}

\begin{document}

\thispagestyle{empty}
\title[Vibrating plate with singular structural damping]{Stabilization for vibrating plate with singular structural damping}
\author{Ka\"{\i}s AMMARI}
\address{Universit\'e de Monastir, Facult\'e des Sciences de Monastir, Analyse et Contr\^ole des EDP, UR 13ES64, Monastir, 5019 Monastir, Tunisia} 
\email{kais.ammari@fsm.rnu.tn} 

\author{Fathi HASSINE}
\address{Universit\'e de Monastir, Facult\'e des Sciences de Monastir, Analyse et Contr\^ole des EDP, UR 13ES64, Monastir, 5019 Monastir, Tunisia}
\email{fathi.hassine@fsm.rnu.tn}

\author{Luc ROBBIANO}
\address{Laboratoire de Math\'ematiques, Universit\'e de Versailles Saint-Quentin en Yvelines, 78035 Versailles, France}
\email{luc.robbiano@uvsq.fr}


\begin{abstract}
We consider the dynamic elasticity equation, modeled by the Euler-Bernoulli plate equation, with a locally distributed singular structural (or viscoelastic
) damping in a boundary domain. Using a frequency domain method combined, based on the Burq's result \cite{burq}, combined with an estimate of Carleman type we provide  precise decay estimate showing that the energy of the system decays logarithmically as the type goes to the infinity.
\end{abstract}


\maketitle

\tableofcontents

\section{Introduction and main results}
\setcounter{equation}{0}
Let $\Omega \subset \R^{n}$, $n \geq 2,$ be a bounded domain with a sufficiently smooth boundary $\partial \Omega=\Gamma=\Gamma_{0}\cup\Gamma_{1}$ such that $\overline{\Gamma}_{0}\cap\overline{\Gamma}_{1}=\emptyset$. Let $\omega$ be an no empty and open subset of $\Omega$ with smooth boundary $\partial\omega=\mathcal{I}\cup\Gamma_{1}$ such that $\overline{\Gamma}_{1}\cap\overline{\mathcal{I}}=\emptyset$ and $\overline{\Gamma}_{0}\cap\overline{\mathcal{I}}=\emptyset$ and (see Figure \ref{figP1}). 

Consider the damping plate system
\begin{equation}
\label{plate1}
\partial_t^2 u + \Delta^2 u - \, \mathrm{div}(a(x) \, \nabla \partial_t u) = 0, \, \Omega \times (0,+\infty), 
\end{equation}
\begin{equation}
\label{plate2}
u = \Delta u = 0, \, \partial \Omega \times (0,+\infty),
\end{equation}
\begin{equation}
\label{plate3}
u(x,0) = u^0, \, \partial_t u(x,0) = u^1 (x), \, \Omega,
\end{equation}
where $a(x)=d\,\mathbb{1}_{\omega}(x)$ and $d>0$ is a constant. This condition ensures that the damping term is singular and effective on the set $\omega$. System \eqref{plate1}-\eqref{plate3}, involving a constructive viscoelastic damping $\mathrm{div}(a(x)\nabla u_{t})$, models the vibrations of an elastic body which has one part made of viscoelastic material. The study of the stabilization of problem involving constructive viscoelastic damping has attached a lot of attention in recent years e.g. \cite{AHR2, AN, AV, ammari-niciase, CCT,chen-liu-liu,hassine1,hassine2,hassine3,liu-liu,liu-rao1,liu-rao2,tebou,tebou1} for the case of the Kelvin-Voigt damping and \cite{DK,MDMKN,tebou2} for the case of the locally distributed structural damping. Noting that the main difference between these two kinds of damping from a mathematical point of view is that the Kelvin-Voigt damping is an operator of the same order of the leading elastic term while the structural order is  of the half of the order of the principal operator. 

The undamped plate equation with $a=0$ occurs as a linear model for vibrating stiff objects where the potential energy involves curvature-like terms which lead to the bi-Laplacian $(-\Delta)^{2}$ as the main ``elastic'' operator. (In the one-dimensional case one obtains the Euler–Bernoulli beam equation). In this model, energy dissipation is neglected and the equation has no smoothing effect as the governing semigroup is unitary on the canonical $L^{2}$-based phase space. One adds damping terms to incorporate the loss of energy. Structural damping describes a situation where higher frequencies are more strongly damped than low frequencies. Here the damping term has ``half of the order'' of the leading elastic term.

From a theoretical point of view, the resulting system can be seen as a transmission problem of mixed type: while the structurally damped plate equation is of parabolic nature, the undamped part is of dissipative nature. Below we will see that the damping is strong enough (independent of the size of the damped part) to obtain logarithm stability for the semigroup of the coupled system. The analogue result for a coupled system of plates was obtained in the study by Denk and Kammerlander \cite{DK} for clamped (Dirichlet) boundary conditions. It is shown in this work that the damping supported near the whole boundary is strong enough to produce uniform exponential decay of the energy of the coupled system. Noting as well the paper of  Denk et al. \cite{MDMKN} in which they consider a transmission problem where a structurally damped plate equation is coupled with a damped or undamped wave equation by transmission conditions. They show that exponential stability holds in the damped-damped situation and polynomial stability (but no exponential stability) holds in the damped-undamped case. However, in this work we deal with damping supported near an arbitrary small part of the boundary. So in particular here we aim to prove the logarithm stabilization of problem \eqref{plate1}-\eqref{plate3}. Our approach consists first to transform the resolvent problem respect to the semigroup operator to a transmission system, then applying a special Carleman estimate adopted to a such coupled system in order to obtain a resolvent estimate with at most exponential growth finally the Burq's result \cite{burq} we find out the decay rate of the energy.
\medskip
 
\begin{figure}[htbp]
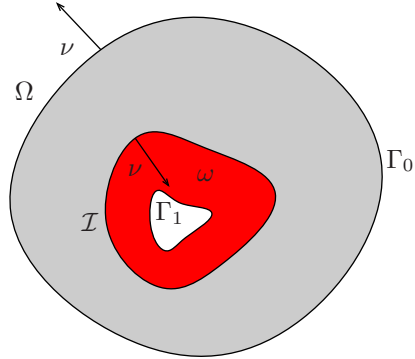

\figinit{pt}
\figpt 0:(10,-15)
\figpt 1:(-50,-10)
\figpt 2:(0,-60)
\figpt 3:(60,-50)
\figpt 4:(90,10)
\figptsym 5:=3/1,4/
\figptsym 6:=2/1,4/
\figpthom 7:=0/1,0.6/
\figpthom 8:=0/2,0.5/
\figpthom 9:=0/3,0.6/
\figpthom 10:=0/4,0.5/
\figpthom 11:=0/5,0.6/
\figpthom 12:=0/6,0.5/
\figpthom 13:=0/1,-0.5/
\figpthom 14:=0/2,-0.5/
\figpthom 15:=0/3,-0.5/
\figpthom 16:=0/4,-0.5/
\figpthom 17:=0/5,-0.5/
\figpthom 18:=0/6,-0.5/
\figpt 19:(10,30)
\figpt 20:(40,-5)
\figpt 21:(5,10)
\figpt 22:(15,-20)
\figpt 23:(45,20)
\figpt 24:(-10,-5)
\figpt 31:(-45,40)
\figpt 32:(-33,70)
\figpt 33:(10,0)
\figpthom 34:=0/7,0.7/
\figpthom 35:=0/8,0.6/
\figpthom 36:=0/9,0.7/
\figpthom 37:=0/10,0.6/
\figpthom 38:=0/11,0.7/
\figpthom 39:=0/12,0.6/
\psbeginfig{}
\psset(width=1)
\psset(fillmode=yes,color=0.8)
\pscurve[1,2,3,4,5,6,1,2,3]
\psset(fillmode=yes,color=2\Redrgb)
\pscurve[7,8,9,10,11,12,7,8,9]
\psset(fillmode=yes,color=\Whitergb)
\pscurve[34,35,36,37,38,39,34,35,36]
\psset(width=0.5)
\psset(color=\Blackrgb)
\psset(fillmode=no)
\pscurve[1,2,3,4,5,6,1,2,3]
\pscurve[7,8,9,10,11,12,7,8,9]
\pscurve[34,35,36,37,38,39,34,35,36]
\psset arrowhead(ratio=0.15)
\psarrow[6,32]
\psarrow[12,33]
\psendfig
\figvisu{\figBoxA}{}{
\figwrites 11:$\omega$(8)
\figwrites 31:$\Omega$(0)
\figwritew 7:$\mathcal{I}$(3)
\figwritee 4:$\Gamma_{0}$(2)
\figwritew 6:$\nu$(10)
\figwrites 12:$\nu$(10)
\figwriten 0:$\Gamma_{1}$(1)
}
\centerline{\box\figBoxA}
\caption{The domain $\Omega$.}
\label{figP1}
\end{figure} 
\medskip
We define the natural energy of $u$ solution of \rfb{plate1}-\rfb{plate3} at instant $t$ by
\begin{equation*}
E(u,t)=
\frac{1}{2} \left(\int_{\Omega}|\partial_{t}u(t,x)|^{2}\,\ud x+\int_{\Omega}|\Delta u(t,x)|^{2}\,\ud x\right), \, \forall \, t \geq 0.
\end{equation*}
Simple formal calculations gives
\begin{equation*}
E(u,0)-E(u,t)= - \, d \,\int_{0}^{t} \int_{\omega}\left|\nabla\partial_{t} u(x,s)\right|^2 \,\ud x\,\ud s,\forall t\geq 0, 
\end{equation*}
and therefore, the energy is non-increasing function of the time variable $t$.

\medskip

\begin{thm}\label{LogStab}
For any $k\in\N^*$ there exists $C>0$ such that for any initial data $(u^{0},u^{1})\in\mathcal{D}(\mathcal{A}^{k})$ the solution $u(x,t)$ of \eqref{plate1} starting from $(u^{0},u^{1})$ satisfying
\be
\label{MR}
E(u,t)\leq\frac{C}{(\ln(2+t))^{2k}}\|(u^{0},u^{1})\|_{\mathcal{D}(\mathcal{A}^{k})}^{2},\quad\forall\,t>0,
\ee
where $(\mathcal{A}, \mathcal{D}(\mathcal{A}))$ is defined in Section \ref{wellposed}.
\end{thm}



\medskip

This paper is organized as follows. In Section \ref{wellposed}, we give the proper functional setting for systems \rfb{plate1}-\rfb{plate3}, then we prove that this system is well-posed and strong stability of the semigroup. In Section \ref{stab}, we study the stabilization for  \rfb{plate1}-\rfb{plate3} by resolvent method and give the explicit decay rate of the energy of the solutions of \rfb{plate1}-\rfb{plate3}.
\section{Well-posedness and strong stability}\label{wellposed}
We define the energy space by $\mathcal{H}= H^2(\Omega) \cap H_{0}^{1}(\Omega)\times L^{2}(\Omega)$ which is endowed with the usual inner product
$$
\left\langle(u_{1},v_{1});(u_{2},v_{2})\right\rangle=\int_{\Omega}\Delta u_{1}(x).\Delta \overline{u}_{2}(x)\,\dd x + \int_{\Omega} v_{1}(x)\overline{v}_{2}(x)\,\dd x.
$$
We next define the linear unbounded operator $\mathcal{A}:\mathcal{D}(\mathcal{A})\subset\mathcal{H}\longrightarrow\mathcal{H}$ by
$$
\mathcal{D}(\mathcal{A})=\{(u,v)\in\mathcal{H}:\;v\in H^2(\Omega) \cap H^1_0(\Omega) ,\;\Delta^2 u-\mathrm{div}(a\nabla v)\in L^{2}(\Omega),\;\Delta u_{|\partial \Omega} = 0 \}
$$
and
$$
\mathcal{A}(u,v)^{t}=(v,-\Delta^2 u+\mathrm{div}(a\nabla v))^{t}
$$
Then, putting $v=\partial_{t} u$, we can write \eqref{plate1}-\eqref{plate3} into the following Cauchy problem
$$
\frac{d}{dt}(u(t),v(t))^{t}=\mathcal{A}(u(t),v(t))^{t},\;(u(0),v(0))=(u^{0}(x),u^{1}(x)).
$$
\begin{thm}
The operator $\mathcal{A}$ generates a $C_{0}$-semigroup of contractions on the energy space $\mathcal{H}$.
\end{thm}
\begin{proof}
Firstly, it is easy to see that for all $(u,v)\in\mathcal{D}(\mathcal{A})$, we have
$$
\mathrm{Re}\left\langle\mathcal{A}(u,v);(u,v)\right\rangle=-\int_{\Omega}a|\nabla v(x)|^{2}\,\dd x,
$$
which show that the operator $\mathcal{A}$ is dissipative.

Next, for any given $(f,g)\in\mathcal{H}$, we solve the equation $\mathcal{A}(u,v)=(f,g)$, which is recast on the following way
\begin{equation}\label{WPplate}
\left\{\begin{array}{l}
v=f,
\\
-\Delta^{2} u+\mathrm{div}(a\nabla f)=g.
\end{array}\right.
\end{equation}
It is well known that by Lax-Milgram's theorem the system \eqref{WPplate} admits a unique solution $u\in H^{2}(\Omega)\cap H_{0}^{1}(\Omega)$. Moreover by multiplying the second line of \eqref{WPplate} by $\overline{u}$ and integrating over $\Omega$ and using Cauchy-Schwarz inequality we find that there exists a constant $C>0$ such that
$$
\int_{\Omega}|\Delta u(x)|^{2}\,\dd x\leq C\left(\int_{\Omega}|\Delta f(x)|^{2}\,\ud x+\int_{\Omega}|g(x)|^{2}\,\ud x\right).
$$
It follows that for all $(u,v)\in\mathcal{D}(\mathcal{A})$ we have
$$
\|(u,v)\|_{\mathcal{H}}\leq C\|(f,g)\|_{\mathcal{H}}.
$$
This imply that $0\in\rho(\mathcal{A})$ and by contraction principle, we easily get $R(\lambda\mathrm{I}-\mathcal{A})=\mathcal{H}$ for sufficient small $\lambda>0$. The density of the domain of $\mathcal{A}$ follows from \cite[Theorem 1.4.6]{Pazy}. Then thanks to Lumer-Phillips Theorem (see \cite[Theorem 1.4.3]{Pazy}), the operator $\mathcal{A}$ generates a $C_{0}$-semigroup of contractions on the Hilbert $\mathcal{H}$. 
\end{proof}
\begin{thm}
The semigroup $e^{t\mathcal{A}}$ is strongly stable in the energy space $\mathcal{H}$, i.e,
$$
\lim_{t\to+\infty}\|e^{t\mathcal{A}}(u_{0},v_{0})^{t}\|_{\mathcal{H}}=0,\;\forall\,(u_{0},v_{0})\in\mathcal{H}.
$$ 
\end{thm}
\begin{proof}
To show that the semigroup $(e^{t\mathcal{A}})_{t\geq 0}$ is strongly stable we only have to prove that the intersection of $\sigma(\mathcal{A})$ with $i\mathbb{R}$ is an empty set. Since the resolvent of the operator $\mathcal{A}$ is not compact (see \cite{liu-liu, liu-rao2}) but $0\in\rho(\mathcal{A})$ we only need to prove that $(i\mu I-\mathcal{A})$ is a one-to-one correspondence in the energy space $\mathcal{H}$ for all $\mu\in\mathbb{R}^{*}$. 

i) Let $(u,v)\in\mathcal{D}(\mathcal{A})$ such that 
\begin{equation}\label{Iplate}
\mathcal{A}(u,v)^{t}=i\mu(u,v)^{t}.
\end{equation}
Then taking the real part of the scalar product of \eqref{Iplate} with $(u,v)$ we get
$$
\mathrm{Re}(i\mu\|(u,v)\|_{\mathcal{H}}^{2})=\mathrm{Re}\left\langle\mathcal{A}(u,v),(u,v)\right\rangle=-\int_{\Omega}a(x)|\nabla v|^{2}\dd x=0.
$$
which implies that
\begin{equation}\label{Dplate}
\nabla v=0 \quad \text{ in }\,\omega.
\end{equation}
Inserting \eqref{Dplate} into \eqref{Iplate}, we obtain
\begin{equation}\label{plateI1}
\left\{\begin{array}{ll}
-\mu^{2}u+\Delta^{2} u=0&\text{in }\Omega,
\\
\nabla u=0&\text{in }\omega
\\
u=\Delta u=0&\text{on }\Gamma.
\end{array}\right.
\end{equation}
We set $w=\Delta u-|\mu|u$ then from \eqref{plateI1} one follows
\begin{equation}\label{plateI2}
\Delta w+|\mu|w=0\quad\text{ in }\,\Omega.
\end{equation}
We denote by $w_{j}=\partial_{x_{j}}w$ and we derive \eqref{plateI2} and the second equation of \eqref{plateI1}, one gets
\begin{equation*}
\left\{\begin{array}{ll}
\Delta w_{j}+|\mu|w_{j}=0&\text{in }\Omega,
\\
w_{j}=0&\text{in }\omega.
\end{array}\right.
\end{equation*}

Hence, from the unique continuation theorem we deduce that $w_{j}=0$ in $\Omega$ and therefore $u_{j}=\partial_{x_{j}}u$ satisfies to the following equation
\begin{equation*}
\Delta u_{j}-|\mu|u_{j}=0\quad\text{ in }\,\Omega.
\end{equation*}
Since $u_{j}\equiv 0$ in $\omega$ once again the unique continuation theorem implies that $u_{j}\equiv 0$ in $\Omega$. Hence, $u$ is constant in $\Omega$ then from the boundary condition $u_{|\Gamma}=0$ we follow that $u\equiv 0$ in $\Omega$.  We have thus proved that $\mathrm{Ker}(i\mu I-\mathcal{A})=0$.
	
ii) Now given $(f,g)\in\mathcal{H}$, we solve the equation 
$$
(\mathcal{A}-i\mu I)(u,v)=(f,g)
$$
Or equivalently,
\begin{equation}\label{Splate}
\left\{\begin{array}{ll}
v=f+i\mu u&\text{in }\Omega
\\
-\Delta^{2}u+i\mu\mathrm{div}(a\nabla u)+\mu^{2}u=g+i\mu f-\mathrm{div}(a\nabla f)&\text{in }\Omega.
\end{array}\right.
\end{equation}
Let's define the operator
$$
\begin{array}{rrll}
A:&\mathcal{D}(A)&\longrightarrow&L^{2}(\Omega)
\\
&u&\longmapsto&\Delta^{2}u
\end{array}
$$
where $\mathcal{D}(A)=\{u\in H^{4}(\Omega): u_{|\Gamma}=\Delta u_{|\Gamma}=0\}$.
It is well known that $A$ a defined positive and self adjoint operator. The square root of the operator $A$ is given by
$$
\begin{array}{rrll}
A^{\frac{1}{2}}:&H^{2}(\Omega)\cap H^{1}_{0}(\Omega)&\longrightarrow&L^{2}(\Omega)
\\
&u&\longmapsto&-\Delta u.
\end{array}
$$
We define the bounded operator $Su=-A^{-1}(\mathrm{div}(a\nabla u))$ in $H^{2}(\Omega)\cap H^{1}_{0}(\Omega)$ and since $S$ is a self-adjoint operator then we have $0\in\rho(I+i\mu S)$.
\\
On the one hand, the second line of \eqref{Splate} can be written as follow
\begin{equation}\label{Eqplate}
(I+i\mu S)u-\mu^{2}A^{-1}u=-A^{-1}\left[g+i\mu f-\mathrm{div}(a\nabla f)\right].
\end{equation}
Let $u\in\mathrm{Ker}(I-\mu^{2}(I+i\mu S)^{-1}A^{-1})$, then $\mu^{2}u-A(I+i\mu S)u=0$ and it is clear that $u\in\mathcal{D}(A)$. It follows that
\begin{equation}\label{Aplate}
\mu^{2}u-\Delta^{2}u+i\mu\mathrm{div}(a\nabla u)=0.
\end{equation}
Multiplying \eqref{Aplate} by $\overline{u}$ and integrating over $\Omega$, then by Green's formula we obtain 
$$
\mu^{2}\int_{\Omega}|u(x)|^{2}\,\dd x-\int_{\Omega}|\Delta u(x)|^{2}\,\dd x-id\mu\int_{\omega}|\nabla u(x)|^{2}\,\dd x=0.
$$
By taking its imaginary part it follows 
$$
d\int_{\omega}|\nabla u(x)|^{2}\,\dd x=0,
$$
and this implies that $\nabla u=0$ in $\omega$. Inserting this last equation into~\eqref{Aplate} we get
$$
\mu^{2}u-\Delta^{2} u=0,\qquad \text{in }\Omega.
$$
Following the steps of the first part of this proof we can prove that $u=0$ and this imply that $\mathrm{Ker}((I-\mu^{2}(I+i\mu S)^{-1}A^{-1})=\{0\}$.
\\
On the other hand, the compactness of the injection $H^{2}(\Omega)\cap H^{1}_{0}(\Omega)\hookrightarrow L^{2}(\Omega)$ implies the compactness of the operator $A^{-\frac{1}{2}}$ and consequently the compactness of the operator $A^{-1}$ as well. Therefore thanks to Fredholm's alternative, the operator $(I-\mu^{2}(I+i\mu S)^{-1}A^{-1})$ is bijective in $L^{2}(\Omega)$. Then by setting
$$
\Lambda u=\mu^{2}A^{-1}u-(I+i\mu S)u=(I+i\mu S)(\mu^{2}(I+i\mu S)^{-1}A^{-1}-I)u.
$$ 
we deduce that $\Lambda$ is a bijection in $H^{2}(\Omega)\cap H^{1}_{0}(\Omega)$. It is not difficult to see that equation of \eqref{Eqplate} is equivalent to the following equation
$$
\Lambda u=A^{-1}(g+i\mu f-\mathrm{div}(a\nabla f)).
$$
So that, equation \eqref{Eqplate} have a unique solution in $H^{2}(\Omega)\cap H^{1}_{0}(\Omega)$ and it is clear that $u\in\mathcal{D}(A)$. This prove that the operator $(i\mu I-\mathcal{A})$ is surjective in the energy space $\mathcal{H}$. 
\\
The proof is thus complete.
\end{proof}
\section{Stabilization result}\label{stab}
In this section, we will prove the logarithmic stability of the system \eqref{plate1}. To this end, we establish a particular resolvent estimate precisely we will show that for some constant $C>0$ we have
\begin{equation}\label{Splate24}
\|(\mathcal{A}-i\mu\,I)^{-1}\|_{\mathcal{L}(\mathcal{H})}\leq C\e^{C|\mu|},\qquad \forall\,|\mu|\gg 1,
\end{equation}
and then by Burq's result \cite{burq} and  the remark of Duyckaerts \cite[section 7]{Duyckaerts} (see also \cite{batty}) we obtain the expected decay rate of the energy.
Let $\mu$ be a real number such that $|\mu|$ is large, and assume that
\begin{equation}\label{Splate1}
(\mathcal{A}-i\mu\,I)(u,v)^{t}=(f,g)^{t},\quad (u,v)\in\mathcal{D}(\mathcal{A}),\quad (f,g)\in\mathcal{H}.
\end{equation}
which can be written as follow
\begin{equation*}
\left\{\begin{array}{ll}
v-i\mu u=f&\text{in }\Omega
\\
-\Delta^{2}u+\mathrm{div}(a(x)\nabla v)-i\mu v=g&\text{in }\Omega,
\end{array}\right.
\end{equation*}
or equivalently,
\begin{equation}\label{Splate2}
\left\{\begin{array}{ll}
v=f+i\mu u&\text{in }\Omega
\\
-\Delta^{2}u+i\mu\,\mathrm{div}(a(x)\nabla u)+\mu^{2}u=g+i\mu f-\mathrm{div}(a(x)\nabla f)&\text{in }\Omega.
\end{array}\right.
\end{equation}
Multiplying the second line of \eqref{Splate2} by $\overline{u}$ and integrating over $\Omega$ then by Green's formula we obtain
\begin{equation}\label{Splate3}
\int_{\Omega}(g+i\mu f)\overline{u}\,\ud x+\int_{\omega}a\nabla f.\nabla\overline{u}\,\ud x=\mu^{2}\int_{\Omega}|u|^{2}\,\ud x-\int_{\Omega}|\Delta u|^{2}\,\ud x-i\mu\int_{\omega}a|\nabla u|^{2}\,\ud x.
\end{equation}
Taking the imaginary part of \eqref{Splate3} we obtain
\begin{align}
\label{Splate4}
|\mu|\int_{\omega}a|\nabla u|^{2}\,\ud x&\leq\left(\int_{\Omega}\left(|g+i\mu f|^{2}\right)\,\ud x\right)^{\frac{1}{2}}.\left(\int_{\Omega}|u|^{2}\,\ud x\right)^{\frac{1}{2}}+\left(\int_{\Omega}|\nabla f|^{2}\,\ud x\right)^{\frac{1}{2}}.\left(\int_{\Omega}|\nabla u|^{2}\,\ud x\right)^{\frac{1}{2}}\nonumber
\\
&\leq \left(\mu^{2}\int_{\Omega}|\Delta f|^{2}\,\ud x+\int_{\Omega}|g|^{2}\,\ud x\right)^{\frac{1}{2}}.\left(\left(\int_{\Omega}|u|^{2}\,\ud x\right)^{\frac{1}{2}}+\left(\int_{\Omega}|\nabla u|^{2}\,\ud x\right)^{\frac{1}{2}}\right)
\end{align}

%
%

Then by setting $u=u_{1}\,\mathbb{1}_{\omega}+ u_{2}\, \mathbb{1}_{\Omega\setminus\bar{\omega}}$, $v=v_{1}\,\mathbb{1}_{\omega}+v_{2}\,\mathbb{1}_{\Omega\setminus\bar{\omega}}$, $f=f_{1}\,\mathbb{1}_{\omega}+f_{2}\, \mathbb{1}_{\Omega \setminus \bar{\omega}}$ and $g=g_{1}\,\mathbb{1}_{\omega}+g_{2}\,\mathbb{1}_{\Omega\setminus\bar{\omega}}$ system \eqref{Splate2} is transformed to the following transmission equation   
\begin{equation}\label{Splate5}
\left\{\begin{array}{ll}
v_{1}=i\mu u_{1}+f_{1}&\text{in }\omega
\\
v_{2}=i\mu u_{2}+f_{2}&\text{in }\Omega\backslash\overline{\omega}
\\
-\Delta^{2}u_{1}+id\mu\Delta u_{1}+\mu^{2}u_{1}=g_{1}+i\mu f_{1}-d\Delta f_{1}&\text{in }\omega
\\
-\Delta^{2}u_{2}+\mu^{2}u_{2}=g_{2}+i\mu f_{2}&\text{in }\Omega\backslash\overline{\omega},
\end{array}\right.
\end{equation}
where the following the transmission conditions
\begin{equation}\label{Splate6}
\left\{\begin{array}{ll}
u_{1}=u_{2}&\text{on }\mathcal{I}
\\
\partial_{\nu}u_{1}=\partial_{\nu}u_{2}&\text{on }\mathcal{I}
\\
\Delta  u_{1}=\Delta u_{2}&\text{on }\mathcal{I}
\\
\partial_{\nu}(\Delta u_{1}-id\mu u_{1}-df_{1})=\partial_{\nu}\Delta u_{2}&\text{on }\mathcal{I},
\end{array}\right.
\end{equation}
follow from the regularity of the state, and with the boundary conditions
\begin{equation}\label{Splate7}
\left\{\begin{array}{ll}
u_{1}=\Delta u_{1}=0&\text{on }\Gamma_{1},
\\
u_{2}=\Delta u_{2}=0&\text{on }\Gamma_{0},
\end{array}\right.
\end{equation}
where $\nu(x)$ denote the outer unit normal to $\Omega\setminus\overline{\omega}$ on $\Gamma_{0}$ and on $\mathcal{I}$ (see Figure \ref{figP1}).

Now we can prove the resolvent estimate \eqref{Splate24}. We set $w_{1}=\Delta u_{1}+(|\mu|-id\mu)u_{1}$ and $w_{2}=\Delta u_{2}+|\mu|u_{2}$, then the system \eqref{Splate5}-\eqref{Splate7} can be recast as follow
\begin{equation}\label{Splate16}
\left\{\begin{array}{ll}
\ds-\Delta w_{1}+|\mu|w_{1}=\Phi_{1}&\text{in }\omega
\\
\ds-\Delta w_{2}+|\mu|w_{2}=\Phi_{2}&\text{in }\Omega\setminus\overline{\omega},
\end{array}\right.
\end{equation}
the transmission conditions
\begin{equation}\label{Splate17}
\left\{\begin{array}{ll}
w_{1}=w_{2}+\phi_{1}&\text{on }\mathcal{I}
\\
\partial_{\nu}w_{1}=\partial_{\nu}w_{2}+\phi_{2}&\text{on }\mathcal{I},
\end{array}\right.
\end{equation}
and the boundary conditions
\begin{equation}\label{Splate18}
\left\{\begin{array}{ll}
w_{1}=0&\text{on }\Gamma_{1}
\\
w_{2}=0&\text{on }\Gamma_{0},
\end{array}\right.
\end{equation}
where we have denoted by $\Phi_{1}=g_{1}+i\mu f_{1}-d\Delta f_{1}-id|\mu|.\mu u_{1}$, $\ds\Phi_{2}=g_{2}+i\mu f_{2}$, $\phi_{1}=-id\mu u_{1}$ and $\phi_{2}=d\partial_{\nu}f_{1}$.

We denoted by $B_{r}$ a ball of radius $r>0$ in $\omega$ and $B_{r}^{c}$ its complementary such that $B_{4r}\subset\omega$. Let's introduce the cut-off function $\chi\in\mathcal{C}^{\infty}(\omega)$ by
$$
\chi(x)=\left\{\begin{array}{ll}
1&\text{in } B_{3r}^{c}
\\
0&\text{in } B_{2r}.
\end{array}\right.
$$
Next, we denote by $\widetilde{w}_{1}=\chi w_{1}$ then from the first line of \eqref{Splate16}, one sees that
\begin{equation}\label{Splate19}
\begin{array}{ll}
-\Delta\widetilde{w}_{1}+|\mu|\widetilde{w}_{1}=\widetilde{\Phi}_{1}&\text{in }\omega,
\end{array}
\end{equation}
where $\widetilde{\Phi}_{1}=\chi\Phi_{1}-[\Delta,\chi]w_{1}$.  We denote by $\Omega_{1}=\omega\setminus\overline{B}_{r}$ and $\Omega_{2}=\Omega\setminus\overline{\omega}$.

Our proof of \eqref{Splate24} is based on a Carleman estimate established in \cite{AHR2} by Ammari, Hassine and Robbiano and recalled here in the following theorem.
\begin{thm}\cite[Theorem 3.2]{AHR2}\label{Pcar5}
Consider a bounded smooth open set $\mathcal{U}$ of $\R^{n}$ with boundary $\partial\mathcal{U}=\gamma$. We set $\mathcal{U}_{1}$ and $\mathcal{U}_{2}$ two smooth open subsets of $\mathcal{U}$ with boundaries $\partial\mathcal{U}_{1}=\gamma_{0}$ and $\partial\mathcal{U}_{2}=\gamma_{0}\cup\gamma$ such that $\overline{\gamma}_{0}\cup\overline{\gamma}=\emptyset$. We denote by $\nu(x)$ the unit outer normal to $\mathcal{U}_{2}$ if $x\in\gamma_{0}\cup\gamma$.

For $\tau$ a large parameter and $\varphi_{1}$ and $\varphi_{2}$ two weight functions of class $\mathcal{C}^{\infty}$ in 
$\overline{\mathcal{U}}_{1}$ and $\overline{\mathcal{U}}_{2}$ respectively such that 
$\varphi_{1|\gamma_{0}}=\varphi_{2|\gamma_{0}}$ we denote by $\varphi(x)=\mathrm{diag}(\varphi_{1}(x),\varphi_{2}(x))$ 
and   let $\alpha$ be a non null complex number. We set the differential operator
$$
P=\mathrm{diag}(P_{1},P_{2})=\mathrm{diag}\left(-\Delta\pm\tau,-\Delta\pm\tau\right),
$$
and its conjugate operator
$$
P(x,D,\tau)=\e^{\tau\varphi}P\e^{-\tau\varphi}=\mathrm{diag}(P_{1}(x,D,\tau),P_{2}(x,D,\tau)),
$$
with principal symbol $p(x,\xi,\tau)$ given by
\begin{align*}
p(x,\xi,\tau)&=\mathrm{diag}(p_{1}(x,\xi,\tau),p_{2}(x,\xi,\tau))
\\
&=\mathrm{diag}(|\xi|^{2}
+2i\tau \xi\nabla\varphi_{1}-\tau^{2}|\nabla\varphi_{1}|^{2},|\xi|^{2}+2i\tau\xi\nabla\varphi_{2}-\tau^{2}|\nabla\varphi_{2}|^{2}).
\end{align*}
We define the tangential operators $\op (B_{1})$  and $\op(B_{2})$ by
\begin{equation}\label{Pcar1}
\op(B_{1})u=u_{1|\gamma_{0}}-u_{2|\gamma_{0}}\qquad\text{and}\qquad\op(B_{2})u=\partial_{\nu}u_{1|\gamma_{0}}-\partial_{\nu}u_{2|\gamma_{0}}.
\end{equation} 

Assume that the weight function $\varphi$ defined on $\mathcal{U}$ satisfies
\begin{eqnarray}
|\nabla\varphi_{k}(x)|> 0,\;\forall\,x\in\overline{\mathcal{U}}_{k},\quad k=1,2,\label{Pcar48}
\\
\partial_{\nu}\varphi_{|\gamma}(x)< 0,\label{Pcar52}
\\
\partial_{\nu}\varphi_{k|\gamma_{0}}(x)>0,\quad k=1,2,
\\
\left(\partial_{\nu}\varphi_{1|\gamma_{0}}(x)\right)^{2}-\left(\partial_{\nu}\varphi_{2|\gamma_{0}}(x)\right)^{2}>1,
\end{eqnarray}
and the sub-ellipticity condition
\begin{equation}\label{Pcar49}
\exists\,c>0,\;\forall\,(x,\xi)\in\overline{\mathcal{U}}_{k}\times\R^{n},\;p_{k}(x,\xi)=0\,\Longrightarrow\,\left\{\mathrm{Re}(p_{k}),\mathrm{Im}(p_{k})\right\}(x,\xi,\tau)\geq c\langle\xi,\tau\rangle^{3}.
\end{equation}
Then there exist $C>0$ and $\tau_{0}>0$ such that we have the following estimate
\begin{align}\label{Pcar4}
&\tau^{3}\|\e^{\tau\varphi}u\|_{L^{2}(\mathcal{U})}^{2}+\tau\|\e^{\tau\varphi}\nabla u\|_{L^{2}(\mathcal{U})}^{2}
\\
&\qquad  \leq C\Big(\|\e^{\tau\varphi}Pu\|_{L^{2}(\mathcal{U})}^{2}+\tau^{2}\|\e^{\tau\varphi}\op(B_{1})u\|_{H^{\frac{1}{2}}(\gamma_{0})}^{2}
+\tau\|\e^{\tau\varphi}\op(B_{2})u\|_{L^{2}(\gamma_{0})}^{2}\Big)   \notag
\end{align}
for all $\tau\geq\tau_{0}$ and $u=(u_{1},u_{2})\in H^{2}(\mathcal{U}_{1})\times H^{2}(\mathcal{U}_{2})$ such that $u_{2|\gamma}=0$.
\end{thm}
Following to \cite{burq} or \cite{hassine2} or \cite{hassine3} we can find four weight functions $\varphi_{1,1}$, $\varphi_{1,2}$, $\varphi_{2,1}$ and $\varphi_{2,2}$, a finite number of points $x_{j,k}^{i}$ where $\overline{B(x_{j,k}^{i},2\varepsilon)}\subset\Omega_{j}$ for all $j,k=1,2$ and $i=1,\ldots,N_{j,k}$ such that $\ds\left[\bigcup_{i=1}^{N_{j,1}}B(x_{j,1}^{i},2\varepsilon)\right]\bigcap\left[\bigcup_{i=1}^{N_{j,2}}B(x_{j,2}^{i},2\varepsilon)\right]=\emptyset$ and by denoting $U_{j,k}=\ds\Omega_{j}\bigcap\left(\bigcup_{i=1}^{N_{j,k}}\overline{B(x_{j,k}^{i},\varepsilon)}\right)^{c}$ the weight function $\varphi_{k}=\mathrm{diag}(\varphi_{1,k},\varphi_{2,k})$ verifying the assumption \eqref{Pcar48}-\eqref{Pcar49} in $U_{1,k}\cup U_{2,k}$ with  $\gamma_{0}=\mathcal{I}$. Moreover, $\varphi_{j,k}<\varphi_{j,k+1}$ in $\ds\bigcup_{i=1}^{N_{j,k}}B(x_{j,k}^{i},2\varepsilon)$ for all $j,k=1,2$ where we have denoted by $\varphi_{j,3}=\varphi_{j,1}$.
\\
Let $\chi_{j,k}$ (for $j,k=1,2$) four cut-off functions equal to $1$ in $\ds\left(\bigcup_{i=1}^{N_{j,k}}B(x_{j,k}^{i},2\varepsilon)\right)^{c}$ and supported in $\ds\left(\bigcup_{i=1}^{N_{j,k}}B(x_{j,k}^{i},\varepsilon)\right)^{c}$ (in order to eliminate the critical points of the weight functions $\varphi_{j,k}$). We set $w_{1,1}=\chi_{1,1}\widetilde{w}_{1}$, $w_{1,2}=\chi_{1,2}\widetilde{w}_{1}$, $w_{2,1}=\chi_{2,1}w_{2}$ and $w_{2,2}=\chi_{2,2}w_{2}$. Then from system \eqref{Splate17} and equations \eqref{Splate7} and \eqref{Splate19}, then for $k=1,2$ we obtain
\begin{equation}\label{Splate20}
\left\{\begin{array}{ll}
-\Delta w_{1,k}+|\mu|w_{1,k}=\Psi_{1,k}&\text{in }\omega
\\
-\Delta w_{2,k}+|\mu|w_{2,k}=\Psi_{2,k}&\text{in }\Omega\setminus\overline{\omega}
\\
w_{1,k}=w_{2,k}+\phi_{1}&\text{on }\mathcal{I}
\\
\partial_{\nu}w_{1,k}= \partial_{\nu} w_{2,k}+\phi_{2}&\text{on }\mathcal{I}
\\
w_{1,k}=0&\text{on }\Gamma_{1}
\\
w_{2,k}=0&\text{on }\Gamma_{0},
\end{array}\right.
\end{equation}
where
\begin{equation}\label{Splate21}
\left\{\begin{array}{l}
\Psi_{1,k}=\chi_{1,k}\widetilde{\Phi}_{1}-[\Delta,\chi_{1,k}]\widetilde{w}_{1}
\\
\Psi_{2,k}=\chi_{2,k}\Phi_{2}-[\Delta,\chi_{2,k}]w_{2}.
\end{array}\right.
\end{equation}
Applying now Carleman estimate \eqref{Pcar4} to the system \eqref{Splate20} with $\tau=|\mu|$ then for $k=1,2$ we have
\begin{multline*}
\tau^{3 }\sum_{j=1,2}\|\e^{\tau\varphi_{j,k}}w_{j,k}\|_{L^{2}(U_{j,k})}^{2}+\tau\sum_{j=1,2}\|\e^{\tau\varphi_{j,k}}\nabla w_{j,k}\|_{L^{2}(U_{j,k})}^{2}
\\
\leq C\Big(\|\e^{\tau\varphi_{1,k}}\Psi_{1,k}\|_{L^{2}(U_{1,k})}^{2}
+\|\e^{\tau\varphi_{2,k}}\Psi_{2,k}\|_{L^{2}(U_{2,k})}^{2}
+\tau^{2}\|\e^{\tau\varphi_{1,k}}\phi_{1}\|_{H^{\frac{1}{2}}(\mathcal{I})}^{2}+\tau\|\e^{\tau\varphi_{1,k}}\phi_{2}\|_{L^{2}(\mathcal{I})}^{2}\Big).
\end{multline*}
From the expression of $\Psi_{1,k}$ and $\Psi_{2,k}$ in \eqref{Splate21}, then we can write
\begin{multline*}
\tau^{3 }\sum_{j=1,2}\|\e^{\tau\varphi_{j,k}}w_{j,k}\|_{L^{2}(U_{j,k})}^{2}+\tau\sum_{j=1,2}\|\e^{\tau\varphi_{j,k}}\nabla w_{j,k}\|_{L^{2}(U_{j,k})}^{2}\leq C\Big(\|\e^{\tau\varphi_{1,k}}\Phi_{1}\|_{L^{2}(U_{1,k})}^{2}
\\
+\|\e^{\tau\varphi_{2,k}}\Phi_{2}\|_{L^{2}(U_{2,k})}^{2}
+\|\e^{\tau\varphi_{1,k}}[\Delta,\chi_{1,k}]\widetilde{w}_{1}\|_{L^{2}(U_{1,k})}^{2}+\|\e^{\tau\varphi_{1,k}}[\Delta,\chi]w_{1}\|_{L^{2}(U_{1,k})}^{2}
\\
+\|\e^{\tau\varphi_{2,k}}[\Delta,\chi_{2,k}]w_{2}\|_{L^{2}(U_{2,k})}^{2}+\tau^{2}\|\e^{\tau\varphi_{1,k}}\phi_{1}\|_{H^{\frac{1}{2}}(\mathcal{I})}^{2}+\tau\|\e^{\tau\varphi_{1,k}}\phi_{2}\|_{L^{2}(\mathcal{I})}^{2}\Big).
\end{multline*}
Adding the two last estimates and using the property of the weight functions $\varphi_{j,1}<\varphi_{j,2}$ in $\ds\bigcup_{i=1}^{N_{j,1}}B(x_{j,1}^{i},2\varepsilon)$ and $\varphi_{j,2}<\varphi_{j,1}$ in $\ds\bigcup_{i=1}^{N_{j,2}}B(x_{j,2}^{i},2\varepsilon)$ for all $j=1,2$, then we can absorb first order the terms $[\Delta,\chi_{1,k}]\widetilde{w}_{1}$ and $[\Delta,\chi_{2,k}]w_{2}$ at the right hand side into the left hand side for $\tau>0$  sufficiently large, mainly we obtain
\begin{multline*}
\qquad\tau^{3}\int_{\Omega_{1}}\left(\e^{2\tau\varphi_{1,1}}+\e^{2\tau\varphi_{1,2}}\right)|\widetilde{w}_{1}|^{2}\,\ud x+\tau^{3}\int_{\Omega_{2}}\left(\e^{2\tau\varphi_{2,1}}+\e^{2\tau\varphi_{2,2}}\right)|w_{2}|^{2}\,\ud x
\\
+\tau\int_{\Omega_{1}}\left(\e^{2\tau\varphi_{1,1}}+\e^{2\tau\varphi_{1,2}}\right)|\nabla\widetilde{w}_{1}|^{2}\,\ud x+\tau\int_{\Omega_{2}}\left(\e^{2\tau\varphi_{2,1}}+\e^{2\tau\varphi_{2,2}}\right)|\nabla w_{2}|^{2}\,\ud x
\\
\leq C\bigg(\int_{\Omega_{1}}\left(\e^{2\tau\varphi_{1,1}}+\e^{2\tau\varphi_{1,2}}\right)|\Phi_{1}|^{2}\,\ud x+\int_{\Omega_{2}}\left(\e^{2\tau\varphi_{2,1}}+\e^{2\tau\varphi_{2,2}}\right)|\Phi_{2}|^{2}\,\ud x
\\
 +\tau^{2}\left(\|\e^{\tau\varphi_{1,1}}\phi_{1}\|_{H^{\frac{1}{2}}(\mathcal{I})}^{2}+\|\e^{\tau\varphi_{1,2}}\phi_{1}\|_{H^{\frac{1}{2}}(\mathcal{I})}^{2}\right)+\tau\left(\|\e^{\tau\varphi_{1,1}}\phi_{2}\|_{L^{2}(\mathcal{I})}^{2}+\|\e^{\tau\varphi_{1,2}}\phi_{2}\|_{L^{2}(\mathcal{I})}^{2}\right)
\\
+\int_{\Omega_{1}}\left(\e^{2\tau\varphi_{1,1}}+\e^{2\tau\varphi_{1,2}}\right)|[\Delta,\chi]w_{1}|^{2}\,\ud x\bigg).
\end{multline*}
Since $\chi\equiv 1$ outside $B_{3r}$ then using the expressions of $\phi_{1}$ and $\phi_{2}$ we obtain
\begin{multline}\label{Splate22}
\tau^{3}\int_{\omega\setminus B_{3r}}\left(\e^{2\tau\varphi_{1,1}}+\e^{2\tau\varphi_{1,2}}\right)|w_{1}|^{2}\,\ud x+\tau^{3}\int_{\Omega\setminus\omega}\left(\e^{2\tau\varphi_{2,1}}+\e^{2\tau\varphi_{2,2}}\right)|w_{2}|^{2}\,\ud x
\\
\tau\int_{\omega\setminus B_{3r}}\left(\e^{2\tau\varphi_{1,1}}+\e^{2\tau\varphi_{1,2}}\right)|\nabla w_{1}|^{2}\,\ud x+\tau\int_{\Omega\setminus\omega}\left(\e^{2\tau\varphi_{2,1}}+\e^{2\tau\varphi_{2,2}}\right)|\nabla w_{2}|^{2}\,\ud x
\\
\leq C\bigg(\int_{\omega}\left(\e^{2\tau\varphi_{1,1}}+\e^{2\tau\varphi_{1,2}}\right)|\Phi_{1}|^{2}\,\ud x+\int_{\Omega\setminus\omega}\left(\e^{2\tau\varphi_{2,1}}+\e^{2\tau\varphi_{2,2}}\right)|\Phi_{2}|^{2}\,\ud x
\\
+\tau^{3}\left(\|\e^{\tau\varphi_{1,1}}u_{1}\|_{H^{\frac{1}{2}}(\mathcal{I})}^{2}+\|\e^{\tau\varphi_{1,2}}u_{1}\|_{H^{\frac{1}{2}}(\mathcal{I})}^{2}\right)+\tau\left(\|\e^{\tau\varphi_{1,1}}\partial_{\nu}f_{1}\|_{L^{2}(\mathcal{I})}^{2}+\|\e^{\tau\varphi_{1,2}}\partial_{\nu}f_{1}\|_{L^{2}(\mathcal{I})}^{2}\right)
\\
+\int_{\Omega_{1}}\left(\e^{2\tau\varphi_{1,1}}+\e^{2\tau\varphi_{1,2}}\right)|[\Delta,\chi]w_{1}|^{2}\,\ud x\bigg).
\end{multline}
Taking the maximum of $\varphi_{1,1}$, $\varphi_{1,2}$, $\varphi_{2,1}$ and $\varphi_{2,2}$ in the right hand side of \eqref{Splate22} and their minimum in the left hand side, next since the operator $[\Delta,\chi]$ is of the first order then by Poincar\'e's inequality, the trace formula and the expressions of $\Phi_{1}$ and $\Phi_{2}$, we follow  
\begin{equation*}
\begin{split}
\|w_{1}\|_{L^{2}(\omega\setminus B_{3r})}^{2}+\|w_{2}\|_{L^{2}(\Omega\setminus\overline{\omega})}^{2}+\|\nabla w_{1}\|_{L^{2}(\omega\setminus B_{3r})}^{2}+\|\nabla w_{2}\|_{L^{2}(\Omega\setminus\overline{\omega})}^{2}
\\
\leq C\e^{C\tau}\Big(\|\nabla w_{1}\|_{L^{2}(\omega)}^{2}+\|f_{1}\|_{L^{2}(\omega)}^{2}+\|\Delta f_{1}\|_{L^{2}(\omega)}^{2}+\|f_{2}\|_{L^{2}(\Omega\setminus\overline{\omega})}^{2}
\\
+\|g_{1}\|_{L^{2}(\omega)}^{2}+\|g_{2}\|_{L^{2}(\Omega\setminus\overline{\omega})}^{2}+\|u_{1}\|_{H^{1}(\omega)}^{2}\Big).
\end{split}
\end{equation*}

Now let $\widetilde{B}_{r}$ a ball of reduce $r$ such that $\overline{B}_{4r}\subset\omega$ and $\overline{B}_{4r}\cap\overline{\widetilde{B}}_{4r}=\emptyset$. We resume the same work with $\widetilde{B}_{r}$ instead of $B_{r}$ we obtain a similar estimate as \eqref{Splate8} namely, one gets
\begin{equation}\label{Splate9}
\begin{split}
\|w_{1}\|_{L^{2}(\omega\setminus\widetilde{B}_{3r})}^{2}+\|w_{2}\|_{L^{2}(\Omega\setminus\overline{\omega})}^{2}+\|\nabla w_{1}\|_{L^{2}(\omega\setminus\widetilde{B}_{3r})}^{2}+\|\nabla w_{2}\|_{L^{2}(\Omega\setminus\overline{\omega})}^{2}
\\
\leq C\e^{C\tau}\Big(\|\nabla w_{1}\|_{L^{2}(\omega)}^{2}+\|f_{1}\|_{L^{2}(\omega)}^{2}+\|\Delta f_{1}\|_{L^{2}(\omega)}^{2}+\|f_{2}\|_{L^{2}(\Omega\setminus\overline{\omega})}^{2}
\\
+\|g_{1}\|_{L^{2}(\omega)}^{2}+\|g_{2}\|_{L^{2}(\Omega\setminus\overline{\omega})}^{2}+\|u_{1}\|_{H^{1}(\omega)}^{2}\Big).
\end{split}
\end{equation}
Summing up the two estimates \eqref{Splate8} and \eqref{Splate9} and using the fact that $\overline{B}_{3r}\cap\overline{\widetilde{B}}_{3r}=\emptyset$, we follow that
\begin{equation}\label{Splate32}
\begin{split}
\|w_{1}\|_{L^{2}(\omega)}^{2}+\|w_{2}\|_{L^{2}(\Omega\setminus\overline{\omega})}^{2}+\|\nabla w_{1}\|_{L^{2}(\omega)}^{2}+\|\nabla w_{2}\|_{L^{2}(\Omega\setminus\overline{\omega})}^{2}
\\
\leq C\e^{C\tau}\Big(\|\nabla w_{1}\|_{L^{2}(\omega)}^{2}+\|f_{1}\|_{L^{2}(\omega)}^{2}+\|\Delta f_{1}\|_{L^{2}(\omega)}^{2}
\\
+\|f_{2}\|_{L^{2}(\Omega\setminus\overline{\omega})}^{2}+\|g_{1}\|_{L^{2}(\omega)}^{2}+\|g_{2}\|_{L^{2}(\Omega\setminus\overline{\omega})}^{2}+\|u_{1}\|_{H^{1}(\omega)}^{2}\Big).
\end{split}
\end{equation}
Noting that $u_{1}$ and $u_{2}$ are solution of the following problem
\begin{equation*}
\left\{\begin{array}{ll}
\ds\Delta u_{1}+|\mu|u_{1}=w_{1}+id|\mu|.\mu u_{1}&\text{in }\omega
\\
\ds\Delta u_{2}+|\mu|u_{2}=w_{2}&\text{in }\Omega\setminus\overline{\omega},
\end{array}\right.
\end{equation*}
the transmission conditions
\begin{equation*}
\left\{\begin{array}{ll}
u_{1}=u_{2}&\text{on }\mathcal{I}
\\
\partial_{\nu}u_{1}=\partial_{\nu}u_{2}&\text{on }\mathcal{I},
\end{array}\right.
\end{equation*}
and the boundary conditions
\begin{equation*}
\left\{\begin{array}{ll}
u_{1}=0&\text{on }\Gamma_{1}
\\
u_{2}=0&\text{on }\Gamma_{0},
\end{array}\right.
\end{equation*}
then as done with $w_{1}$ and $w_{2}$ we can apply Carleman estimate to $u_{1}$ and $u_{2}$ and we get an estimate of the same kind as \eqref{Splate32}, namely we have
\begin{equation*}
\begin{split}
\|u_{1}\|_{L^{2}(\omega)}^{2}+\|u_{2}\|_{L^{2}(\Omega\setminus\overline{\omega})}^{2}+\|\nabla u_{1}\|_{L^{2}(\omega)}^{2}+\|\nabla u_{2}\|_{L^{2}(\Omega\setminus\overline{\omega})}^{2}
\\
\leq C\e^{C|\mu|}\Big(\|\nabla u_{1}\|_{L^{2}(\omega)}^{2}+\|w_{1}\|_{L^{2}(\omega)}^{2}+\|w_{2}\|_{L^{2}(\Omega\setminus\overline{\omega})}^{2}\Big),
\end{split}
\end{equation*}
which imply in particular that
\begin{equation}\label{Splate35}
\begin{split}
\|\nabla u_{2}\|_{L^{2}(\Omega\setminus\overline{\omega})}^{2}\leq C\e^{C|\mu|}\Big(\|\nabla u_{1}\|_{L^{2}(\omega)}^{2}+\|w_{1}\|_{L^{2}(\omega)}^{2}+\|w_{2}\|_{L^{2}(\Omega\setminus\overline{\omega})}^{2}\Big).
\end{split}
\end{equation}
From \eqref{Splate16}, performing now the following calculation
\begin{align*}
\|\nabla w_{1}\|_{L^{2}(\omega)}^{2}+\|\nabla w_{2}\|_{L^{2}(\omega)}^{2}&=-\left\langle\Delta w_{1},w_{1}\right\rangle_{L^{2}(\omega)}-\left\langle\Delta w_{2},w_{2}\right\rangle_{L^{2}(\Omega\setminus\omega)}
\\
&-\left\langle\partial_{\nu}w_{2},w_{2}\right\rangle_{L^{2}(\mathcal{I})}+\left\langle\partial_{\nu}w_{1},w_{1}\right\rangle_{L^{2}(\mathcal{I})}
\\
&=\left\langle\Phi_{1},w_{1}\right\rangle_{L^{2}(\omega)}+\left\langle\Phi_{2},w_{2}\right\rangle_{L^{2}(\Omega\setminus\omega)}-|\mu|\left(\|w_{1}\|_{L^{2}(\omega)}^{2}+\|w_{2}\|_{L^{2}(\Omega\setminus\omega)}^{2}\right)
\\
&-\left\langle\partial_{\nu}w_{2},w_{2}\right\rangle_{L^{2}(\mathcal{I})}+\left\langle\partial_{\nu}w_{1},w_{1}\right\rangle_{L^{2}(\mathcal{I})}.
\end{align*}
Using the transmission conditions \eqref{Splate17} we obtain
\begin{align*}
\left\langle\partial_{\nu}w_{1},w_{1}\right\rangle_{L^{2}(\mathcal{I})}-\left\langle\partial_{\nu}w_{2},w_{2}\right\rangle_{L^{2}(\mathcal{I})}&=id\mu\left\langle\partial_{\nu}w_{1},u_{1}\right\rangle_{L^{2}(\mathcal{I})}+d\left\langle\partial_{\nu}f_{1},w_{1}+id\mu u_{1}\right\rangle_{L^{2}(\mathcal{I})}
\\
&=-id\mu\left(\left\langle\Delta w_{1},u_{1}\right\rangle_{L^{2}(\omega)}+\left\langle\nabla w_{1},\nabla u_{1}\right\rangle_{L^{2}(\omega)}\right)
\\
&+d\left\langle\partial_{\nu}f_{1},w_{1}+id\mu u_{1}\right\rangle_{L^{2}(\mathcal{I})}
\\
&=-id\mu\Big(|\mu|\left\langle w_{1},u_{1}\right\rangle_{L^{2}(\omega)}+\left\langle\nabla w_{1},\nabla u_{1}\right\rangle_{L^{2}(\omega)}
\\
&-\left\langle\Phi_{1},u_{1}\right\rangle_{L^{2}(\omega)}\Big)+d\left\langle\partial_{\nu}f_{1},w_{1}+id\mu u_{1}\right\rangle_{L^{2}(\mathcal{I})}.
\end{align*}
Putting together the two last equalities we find
\begin{align}\label{Splate34}
\|\nabla w_{1}\|_{L^{2}(\omega)}^{2}+\|\nabla w_{2}\|_{L^{2}(\omega)}^{2}&=\left\langle\Phi_{1},w_{1}\right\rangle_{L^{2}(\omega)}+\left\langle\Phi_{2},w_{2}\right\rangle_{L^{2}(\Omega\setminus\omega)}-|\mu|\left(\|w_{1}\|_{L^{2}(\omega)}^{2}+\|w_{2}\|_{L^{2}(\Omega\setminus\omega)}^{2}\right)\nonumber
\\
&-id\mu\left(|\mu|\left\langle w_{1},u_{1}\right\rangle_{L^{2}(\omega)}+\left\langle\nabla w_{1},\nabla u_{1}\right\rangle_{L^{2}(\omega)}-\left\langle \Phi_{1},u_{1}\right\rangle_{L^{2}(\omega)}\right)\nonumber
\\
&+d\left\langle\partial_{\nu}f_{1},w_{1}\right\rangle_{L^{2}(\mathcal{I})}-id^{2}\mu\left\langle\partial_{\nu}f_{1},u_{1}\right\rangle_{L^{2}(\mathcal{I})}.
\end{align}
The Poincar\'e inequality, the trace formula and the Young's inequality imply
\begin{align}\label{Splate8}
&|\mu|\left(\|w_{1}\|_{L^{2}(\omega)}^{2}+\|w_{2}\|_{L^{2}(\Omega\setminus\omega)}^{2}\right)+\|\nabla w_{1}\|_{L^{2}(\omega)}^{2}+\|\nabla w_{2}\|_{L^{2}(\Omega\setminus\omega)}^{2}\nonumber
\\
&\leq C\left(\|\Phi_{1}\|_{L^{2}(\omega)}^{2}+\|\Phi_{2}\|_{L^{2}(\Omega\setminus\omega)}^{2}+|\mu|^{4}.\|\nabla u_{1}\|_{L^{2}(\omega)}^{2}+\|f_{1}\|_{H^{2}(\omega)}^{2}\right)\nonumber
\\
&\leq C\left(\mu^{2}\left(\|f_{1}\|_{H^{2}(\omega)}^{2}+\|f_{2}\|_{H^{2}(\Omega\setminus\omega)}^{2}\right)+\|g_{1}\|_{L^{2}(\omega)}^{2}+\|g_{2}\|_{L^{2}(\Omega\setminus\omega)}^{2}+|\mu|^{4}.\|\nabla u_{1}\|_{L^{2}(\omega)}^{2}\right).
\end{align}
Combining \eqref{Splate32} and \eqref{Splate8} we follow
\begin{equation*}
\begin{split}
\|w_{1}\|_{L^{2}(\omega)}^{2}+\|w_{2}\|_{L^{2}(\Omega\setminus\overline{\omega})}^{2}+C|\mu|\e^{C|\mu|}\left(\|w_{1}\|_{L^{2}(\omega)}^{2}+\|w_{2}\|_{L^{2}(\Omega\setminus\omega)}^{2}\right)
\\
\leq C_{1}\e^{C_{1}|\mu|}\Big(\|f_{1}\|_{H^{2}(\omega)}^{2}+\|f_{2}\|_{H^{2}(\Omega\setminus\omega)}^{2}+\|g_{1}\|_{L^{2}(\omega)}^{2}+\|g_{2}\|_{L^{2}(\Omega\setminus\omega)}^{2}+\|\nabla u_{1}\|_{L^{2}(\omega)}^{2}\Big).
\end{split}
\end{equation*}
From this last estimates and \eqref{Splate35} we find
\begin{equation}\label{Splate36}
\begin{split}
\|w_{1}\|_{L^{2}(\omega)}^{2}+\|w_{2}\|_{L^{2}(\Omega\setminus\omega)}^{2}+2|\mu|.\|\nabla u_{2}\|_{L^{2}(\Omega\setminus\omega)}^{2}\leq C_{1}\e^{C_{1}|\mu|}\Big(\|f_{1}\|_{H^{2}(\omega)}^{2}
\\
+\|f_{2}\|_{H^{2}(\Omega\setminus\omega)}^{2}+\|g_{1}\|_{L^{2}(\omega)}^{2}+\|g_{2}\|_{L^{2}(\Omega\setminus\omega)}^{2}+\|\nabla u_{1}\|_{L^{2}(\omega)}^{2}\Big).
\end{split}
\end{equation}
Evoking $u_{1}$ and $u_{2}$ through the expressions of $w_{1}$ and $w_{2}$ and using the transmission conditions \eqref{Splate6} and the boundary conditions \eqref{Splate7} to perform the following integration by parts
\begin{align*}
\|w_{1}\|_{L^{2}(\omega)}^{2}+\|w_{2}\|_{L^{2}(\Omega\setminus\omega)}^{2}=-2d\mu\,\im\langle \partial_{\nu}u_{1},u_{1}\rangle_{L^{2}(\mathcal{I})}+\|\Delta u_{1}\|_{L^{2}(\omega)}^{2}+\|\Delta u_{2}\|_{L^{2}(\Omega\setminus\omega)}^{2}
\\
+|\mu|^{2}\left((1+d^{2})\|u_{1}\|_{L^{2}(\omega)}^{2}+\|u_{2}\|_{L^{2}(\Omega\setminus\omega)}^{2}\right)-2|\mu|\left(\|\nabla u_{1}\|_{L^{2}(\omega)}^{2}+\|\nabla u_{2}\|_{L^{2}(\Omega\setminus\omega)}^{2}\right).
\end{align*}
From the Young's inequality we obtain
\begin{align}\label{Splate33}
\|w_{1}\|_{L^{2}(\omega)}^{2}+\|w_{2}\|_{L^{2}(\Omega\setminus\omega)}^{2}&\geq \|\Delta u_{1}\|_{L^{2}(\omega)}^{2}+\|\Delta u_{2}\|_{L^{2}(\Omega\setminus\omega)}^{2}
\\
&+|\mu|^{2}\left((1+d^{2})\|u_{1}\|_{L^{2}(\omega)}^{2}+\|u_{2}\|_{L^{2}(\Omega\setminus\omega)}^{2}\right)\nonumber
\\
&-2|\mu|\left(\|\nabla u_{1}\|_{L^{2}(\omega)}^{2}+\|\nabla u_{2}\|_{L^{2}(\Omega\setminus\omega)}^{2}\right)\nonumber
\\
&-\frac{d^{2}|\mu|^{2}}{\varepsilon}\|u_{1}\|_{H^{1}(\omega)}^{2}-\varepsilon\|u_{1}\|_{H^{2}(\omega)}^{2}.\nonumber
\end{align}
Combining \eqref{Splate36} and \eqref{Splate33}, taking $\varepsilon$ small enough and using the Poincar\'e inequality, one gets
\begin{align*}
\|\Delta u_{1}\|_{L^{2}(\omega)}^{2}+\|\Delta u_{2}\|_{L^{2}(\Omega\setminus\omega)}^{2}\leq C\e^{C|\mu|}\Big(\|\Delta f_{1}\|_{L^{2}(\omega)}^{2}+\|\Delta f_{2}\|_{L^{2}(\Omega\setminus\omega)}^{2}
\\
+\|g_{1}\|_{L^{2}(\omega)}^{2}+\|g_{2}\|_{L^{2}(\Omega\setminus\omega)}^{2}+\|\nabla u_{1}\|_{L^{2}(\omega)}^{2}\Big),
\end{align*}
which implies
\begin{align}\label{Splate23}
\|\Delta u\|_{L^{2}(\Omega)}^{2}\leq C\e^{C|\mu|}\left(\|\Delta f\|_{L^{2}(\Omega)}^{2}+\|g\|_{L^{2}(\Omega)}^{2}+\|\nabla u\|_{L^{2}(\omega)}^{2}\right).
\end{align}
Using \eqref{Splate4} and \eqref{Splate23} we follow
\begin{align*}
\|\Delta u\|_{L^{2}(\Omega)}^{2}\leq C\e^{C|\mu|}\Big(\|\Delta f\|_{L^{2}(\Omega)}^{2}+\|g\|_{L^{2}(\Omega)}^{2}+
\\
\left(\|\Delta f\|_{L^{2}(\Omega)}\!+\!\|g\|_{L^{2}(\Omega)}\right)\left(\| u\|_{L^2(\Omega)}+\| \nabla u\|_{L^2(\Omega)}\right)\Big).
\end{align*}
By Poincar\'e inegalit\'e, one has
\begin{equation}\label{Splate25}
\|\Delta u\|_{L^{2}(\Omega)}^{2}\leq C\e^{C|\mu|}\left(\|\Delta f\|_{L^{2}(\Omega)}^{2}+\|g\|_{L^{2}(\Omega)}^{2}
\right).
\end{equation}
We refer to the expression of $v$ in the first line of \eqref{Splate2} and using the fact that 
$$
\|u\|_{L^{2}(\Omega)}\leq C\|\Delta u\|_{L^{2}(\Omega)}
$$
then estimate \eqref{Splate25} gives
\begin{equation}\label{Splate26}
\|v\|_{L^{2}(\Omega)}^{2}\leq C\e^{C|\mu|}\left(\|\Delta f\|_{L^{2}(\Omega)}^{2}+\|g\|_{L^{2}(\Omega)}^{2}\right).
\end{equation}
So that, the estimate \eqref{Splate24} is obtained by the combination of the two estimates \eqref{Splate25} and \eqref{Splate26}. And this completes the proof.

\end{document}